\documentclass{amsart}

\usepackage{a4wide}
\usepackage[initials]{amsrefs}
\usepackage{url}
\usepackage{color}
\usepackage[utf8]{inputenc}
\usepackage{amsmath,amssymb,amsthm}
\usepackage{float}
\usepackage{enumerate}
\usepackage{nicefrac}
\usepackage{graphicx}

\newtheorem{theorem}{Theorem}
\newtheorem{lemma}{Lemma}
\newtheorem{corollary}{Corollary}
\newtheorem{remark}{Remark}

\DeclareMathOperator{\lcm}{lcm}
\DeclareMathOperator{\ord}{ord}

\newtheorem{innercustomlem}{Lemma}
\newenvironment{customlem}[1]
  {\renewcommand\theinnercustomlem{#1}\innercustomlem}
  {\endinnercustomlem}

\title[Sums of unit fractions]{The number of solutions of the Erd\H{o}s-Straus Equation and sums of $k$ unit fractions}

\author[C. Elsholtz]{Christian Elsholtz}
\address[C. Elsholtz]{
Graz University of Technology, Institute of Analysis and Number Theory, Kopernikusgasse 24/II, 8010 Graz, Austria}
\email{elsholtz@math.tugraz.at}

\author[S. Planitzer]{Stefan Planitzer}
\address[S. Planitzer]{
Graz University of Technology, Institute of Analysis and Number Theory, Kopernikusgasse 24/II, 8010 Graz, Austria}
\email{planitzer@math.tugraz.at}

\subjclass[2010]{Primary: 11D68, 
                 Secondary: 11D72 
}

\begin{document}

\begin{abstract}
We prove new upper bounds for the number of representations of an arbitrary rational number as a sum of three unit fractions. In particular, for fixed $m$ there are at most $\mathcal{O}_{\epsilon}(n^{\nicefrac{3}{5}+\epsilon})$ solutions of $\frac{m}{n}=\frac{1}{a_1}+\frac{1}{a_2}+\frac{1}{a_3}$. This improves upon a result of Browning and Elsholtz (2011) and extends a result of Elsholtz and Tao (2013) who proved this when $m=4$ and $n$ is a prime. Moreover there exists an algorithm finding all solutions in expected running time $\mathcal{O}_{\epsilon}\left(n^{\epsilon}\left(\frac{n^3}{m^2}\right)^{\nicefrac{1}{5}}\right)$, for any $\epsilon >0$. We also improve a bound on the maximum number of representations of a rational number as a sum of $k$ unit fractions. Furthermore, we also improve lower bounds. In particular we prove that for given $m\in \mathbb{N}$ in every reduced residue class $e \bmod f$ there exist infinitely many primes $p$ such that the number of solutions of the equation $\frac{m}{p}=\frac{1}{a_1}+\frac{1}{a_2}+\frac{1}{a_3}$ is $\gg_{f,m} \exp\left(\left(\frac{5\log 2}{12 \lcm(m,f)}+o_{f,m}(1)\right)\frac{\log p}{\log \log p}\right)$. Previously the best known lower bound of this type was of order $(\log p)^{0.549}$.
\end{abstract}

\maketitle

\section{Introduction}

We consider the problem of finding upper bounds for the number of solutions in positive integers $a_1$, $a_2$ and $a_3$ of equations of the form
\begin{equation} \label{eq: 3 fractions main equation}
\frac{m}{n}=\frac{1}{a_1}+\frac{1}{a_2}+\frac{1}{a_3}
\end{equation}
where $m,n \in \mathbb{N}$ are fixed. In the case when $m=4$ we call equation~\eqref{eq: 3 fractions main equation} Erd\H{o}s-Straus equation. The Erd\H{o}s-Straus conjecture states that this equation has at least one solution for any $n>1$ (see \cite{CountingThe} and \cite{UnsolvedProblems}*{D11} for classical results concerning the Erd\H{o}s-Straus equation and several related problems, as well as~\cite{PaulErdoes} for a survey of the work of Erd\H{o}s on egyptian fractions). Also the more general equation
\begin{equation} \label{eq: k fractions main equation}
\frac{m}{n}=\sum_{i=1}^k\frac{1}{a_i}
\end{equation}
for $m,n \in \mathbb{N}$ fixed and $a_1, \ldots, a_k \in \mathbb{N}$ received some attention. Browning and Elsholtz~\cite{TheNumber} found upper bounds for the number of solutions of \eqref{eq: k fractions main equation}. For the special case $m=n=1$ they were able to improve a result of S{\'a}ndor~\cite{OnThe} and proved that there are at most $c_0^{(\nicefrac{5}{24}+\epsilon)2^{k}}$ representations of $1$ as a sum of $k$ unit fractions, for any $\epsilon >0$ and sufficiently large $k$. Here $c_0=\lim_{n \rightarrow \infty}u_n^{2^{-n}}=1.264\ldots$ where $u_1=1$ and $u_{n+1}=u_n(u_n+1)$. On the other hand Konyagin~\cite{DoubleExponential} proved a lower bound of order $\exp\left(\exp\left(\left(\frac{(\log 2)(\log 3)}{3}+o(1)\right)\frac{k}{\log k}\right)\right)$ for the number of these representations with distinct denominators. While the Erd\H{o}s-Straus conjecture is about representing certain rational numbers as a sum of just three unit fractions, Martin~\cite{DenseEgyptian} worked on representations of positive rationals as sums of many unit fractions. In particular he proved that every positive rational number $r$ has a representation of the form $r=\sum_{s\in S}\frac{1}{s}$, where the set $S$ contains a positive proportion of the integers less than any sufficiently large real number $x$.      

Chen et.al.~\cite{EgyptianFractions} dealt with representations of $1$ as a sum of $k$ distinct unit fractions where the denominators satisfy certain restrictions (like all of them being odd). Several results on representations of rational numbers as a sum of unit fractions with restrictions on the denominators can be found in the work of Graham~\cites{OnFinite, OnFiniteSums, PaulErdoes}. Elsholtz~\cite{EgyptianFractionsWith} proved a lower bound of similar order as the one of Konyagin for the number of representations of $1$ as a sum of $k$ distinct unit fractions with odd denominators.

For sums of $k$ unit fractions we adopt the notation of~\cite{TheNumber} and define $f_k(m,n)$ to be the number of solutions $(a_1, a_2, \ldots, a_k) \in \mathbb{N}^k$ of equation \eqref{eq: k fractions main equation} with $a_1 \leq a_2 \leq \ldots \leq a_k$, i.e.
$$f_k(m,n)=\left|\left\{(a_1,a_2,\ldots, a_k)\in \mathbb{N}^k:\frac{m}{n}=\frac{1}{a_1}+\frac{1}{a_2}+\cdots+\frac{1}{a_k}, a_1 \leq a_2 \ldots \leq a_k \right\}\right|.$$
Concerning equation \eqref{eq: 3 fractions main equation} with $m=4$ the results of Elsholtz and Tao~\cite{CountingThe} show that the number of solutions $f_3(4,n)$ is related to some divisor questions and is on average a power of $\log n$ (at least when $n$ is prime). It even seems possible that for fixed $m \in \mathbb{N}$ and any $\epsilon>0$ the number of representations of $\frac{m}{n}$ as a sum of $k$ unit fractions is bounded by $\mathcal{O}_{k,\epsilon}(n^{\epsilon})$. More details on this are informally and heuristically discussed in Section~\ref{sec: heuristics}. For general $m$ and $n$ the best known upper bound on the number of solutions of \eqref{eq: 3 fractions main equation} is due to Browning and Elsholtz~\cite{TheNumber}*{Theorem 2} who proved an upper bound of order $\mathcal{O}_{\epsilon}(n^{\epsilon}\left(\frac{n}{m}\right)^{\nicefrac{2}{3}})$. In the case of the Erd\H{o}s-Straus equation with $n=p$ prime Elsholtz and Tao~\cite{CountingThe}*{Proposition 1.7} have improved this bound to $\mathcal{O}_{\epsilon}(p^{\nicefrac{3}{5}+\epsilon})$. It is known that this type of question is easier to study, when the denominator is prime.

Our main result will be the following theorem which provides an upper bound on the number of solutions of equation \eqref{eq: 3 fractions main equation}.

\begin{theorem} \label{thm: main theorem}
For any $m,n \in \mathbb{N}$ and any $\epsilon > 0$ there are at most $\mathcal{O}_{\epsilon}\left(n^{\epsilon}\left(\frac{n^3}{m^2}\right)^{\nicefrac{1}{5}}\right)$ solutions of the equation
$$\frac{m}{n} =\frac{1}{a_1}+\frac{1}{a_2}+\frac{1}{a_3}$$
in positive integers $a_1$, $a_2$ and $a_3$.
\end{theorem}

Note that this improves upon the bound of Browning and Elsholtz in the range $m \ll n^{\nicefrac{1}{4}}$. As a corollary we get that the Elsholtz-Tao bound for the number of solutions of the Erd\H{o}s-Straus equation is true for arbitrary denominators $n \in \mathbb{N}$.

\begin{corollary}
The Erd\H{o}s-Straus equation
$$\frac{4}{n}=\frac{1}{a_1}+\frac{1}{a_2}+\frac{1}{a_3}$$
has at most $\mathcal{O}_{\epsilon}(n^{\nicefrac{3}{5}+\epsilon})$ solutions in positive integers $a_1$, $a_2$ and $a_3$.
\end{corollary}

We also prove the following algorithmic version of Theorem~\ref{thm: main theorem} with a matching upper bound for the expected running time\footnote{For a definition of expected running time see the proof of this corollary at the end of section~\ref{sec: 3 unit fractions}.}.

\begin{corollary} \label{cor: algorithmic aspects}
There exists an algorithm with an expected running time of order $\mathcal{O}_{\epsilon}\left(n^{\epsilon}\left(\frac{n^3}{m^2}\right)^{\nicefrac{1}{5}}\right)$, for any $\epsilon >0$, which lists all representations of the rational number $\frac{m}{n}$ as a sum of three unit fractions. Furthermore all representations of $\frac{m}{n}$ as a sum of $k>3$ unit fractions may be found in expected time $\mathcal{O}_{\epsilon,k}\left(n^{2^{k-3}(\nicefrac{8}{5}+\epsilon)-1}\right)$, for any $\epsilon >0$. 
\end{corollary}

For sums of $k$ unit fractions we will prove the following result.
\begin{theorem} \label{thm: k fractions theorem}
We have
$$f_4(m,n) \ll_{\epsilon}n^{\epsilon}\left(\frac{n^{\nicefrac{4}{3}}}{m^{\nicefrac{2}{3}}}+\frac{n^{\nicefrac{28}{17}}}{m^{\nicefrac{8}{5}}}\right)$$
and for any $k \geq 5$
$$f_k(m,n)\ll_{\epsilon} (kn)^{\epsilon}\left(\frac{k^{\nicefrac{4}{3}}n^2}{m}\right)^{\nicefrac{28}{17}\cdot 2^{k-5}}.$$
\end{theorem}
Keeping in mind that $\frac{28}{17}=1.64705\ldots$, Theorem~\ref{thm: k fractions theorem} may be compared with the following bounds from \cite{TheNumber}*{Theorem 3}:
\begin{align*}
f_4(m,n)&\ll_{\epsilon} n^{\epsilon}\left(\frac{n^{\nicefrac{4}{3}}}{m^{\nicefrac{2}{3}}}+\left(\frac{n}{m}\right)^{\nicefrac{5}{3}}\right),\\
f_k(m,n)&\ll_{\epsilon} (kn)^{\epsilon}\left(\frac{k^{\nicefrac{4}{3}}n^2}{m}\right)^{\nicefrac{5}{3} \cdot2^{k-5}} \text{, for } k\geq 5.
\end{align*}

A well studied special case of Theorem~\ref{thm: k fractions theorem} concerns representations of $1$ as a sum of $k$ unit fractions. Browning and Elsholtz~\cite{TheNumber} mention several related problems which are studied in the literature and can be improved using better upper bounds on $f_k(m,n)$. We summarize these results in the following corollary.

\begin{corollary}
\begin{enumerate}
\item For any $\epsilon >0$ we have that 
$$f_k(1,1)\ll_{\epsilon}k^{\nicefrac{7}{51}\cdot 2^{k-1}+\epsilon}.$$
\item Let $u_n$ be the sequence recursively defined by $u_0=1$ and $u_{n+1}=u_n(u_n+1)$ and set  $c_0=\lim_{n \rightarrow \infty}u_n^{2^{-n}}$. Then for $\epsilon >0$ and $k\geq k(\epsilon)$ we have
$$f_k(1,1) < c_0^{(\nicefrac{7}{17}+\epsilon)2^{k-1}}.$$
\item For $\epsilon>0$ and $k\geq k(\epsilon)$ the number $S(k)$ of positive integer solutions of the equation
$$1=\sum_{i=1}^k\frac{1}{a_i}+\frac{1}{\prod_{i=1}^ka_i}$$
is bounded from above by $c_0^{(\nicefrac{7}{17}+\epsilon)2^k}$.
\end{enumerate}
\end{corollary}
\begin{proof}
The first assertion is an immediate consequence of Theorem~\ref{thm: k fractions theorem}. For the proof of the second statement we refer the reader to the proof of Theorem 4 in~\cite{TheNumber}. The only change necessary is plugging in the bound from Theorem~\ref{thm: k fractions theorem} instead of~\cite{TheNumber}*{Theorem 3} for the last $5$ lines of the proof which amounts to just exchanging one exponent. The last statement follows from the first one and the observation that $S(k)\leq f_{k+1}(1,1)$.
\end{proof}

We note that the number of solutions of the equation $1=\sum_{1=1}^k\frac{1}{a_i}+\frac{1}{\prod_{i=1}^ka_i}$ has applications to problems considered in~\cite{OnTheDiophantine}.

Finally we deal with lower bounds. In~\cite{CountingThe}*{Theorem 1.8} it is shown that we have 
$$f_3(4,n)\geq \exp\left((\log 3+o(1))\frac{\log n}{\log \log n}\right)$$ 
for infinitely many $n \in \mathbb{N}$ and that 
$$f_3(4,n) \geq \exp\left(\left(\frac{\log 3}{2}+o(1)\right)\log \log n\right)$$ for all integers $n$ in a subset of the positive integers with density $1$. The following theorem gives an improvement of these bounds which also give a limitation on improving the upper bounds for the number of solution of the Erd\H{o}s-Straus equation and in the general case. For comparison we note that $\log 3=1.09861\ldots$, $\frac{ \log 3}{2}=0.54930\ldots$ and $\log 6=1.79175\ldots$.

\begin{theorem} \label{thm: lower bounds theorem}
For given $m \in \mathbb{N}$ there are infinitely many $n \in \mathbb{N}$ such that
$$f_3(m,n) \geq \exp\left((\log 6 +o_m(1))\frac{\log n}{\log \log n}\right).$$
Furthermore, for given $m \in \mathbb{N}$, there exists a subset $\mathcal{M}_1$ of the integers with density one, such that for any $n \in \mathcal{M}_1$
\begin{align*}
f_3(m,n)&\geq \left(\frac{1}{\varphi(m)}+o(1)\right)\exp\left((\log 3+o_m(1))\log \log n\right) \cdot \log \log n \\ 
&\gg (\log n)^{\log 3+o_m(1)}.
\end{align*}
For the special case $m=4$ and for integers $n$ in a set $\mathcal{M}_2\subset
\mathbb{N}$ with density one, the last bound may be improved to
$$f_3(4,n)\geq \exp\left((\log 6+o(1))\log \log n\right).$$
\end{theorem}

\begin{remark}
Previous proofs of lower bounds of similar type as the ones in Theorem~\ref{thm: lower bounds theorem} constructed solutions from factorizations of $n$. We get our improvement from additionally taking into account factorizations of a lot of shifts of $n$. Hence our proof also shows that there are many values $a_1$ admitting many pairs $(a_2,a_3)$. Here `many' means $\exp\left((C+o_m(1))\frac{\log n}{\log \log n}\right)$, where the constant $C$ depends on which of the three lower bounds in Theorem~\ref{thm: lower bounds theorem} we consider. 
\end{remark}

We may ask if a lower bound on $f_3(m,n)$ of the first type in Theorem~\ref{thm: lower bounds theorem} does not only hold for infinitely many positive integers $n$ but also for infinitely many prime denominators $p$. In~\cite{CountingThe} there was no lower bound of this type, but it was proved that $f_3(4,p)\gg (\log p)^{0.549}$ for almost all primes. We note that this result implies, using Dirichlet's theorem on primes, the following corollary.

\begin{corollary}
For every reduced residue class $e \bmod f$, i.e. $\gcd(e,f)=1$, there are infinitely many primes $p$ such that $f_3(4,p) \gg (\log p)^{0.549}$, and $p \equiv e \bmod f$. 
\end{corollary}

Here we improve this corollary considerably.

\begin{theorem} \label{thm: lower bound for primes theorem}
For every $m \in \mathbb{N}$ and every reduced residue class $e \bmod f$ there are infinitely many primes $p\equiv e \bmod f$ such that
$$f_3(m,p)\gg_{f,m} \exp\left(\left(\frac{5\log 2}{12 \lcm(m,f)}+o_{f,m}(1)\right)\frac{\log p}{\log \log p}\right).$$
Here $o_{f,m}(1)$ denotes a quantity depending on $f$ and $m$ which goes to zero as $p$ tends to infinity.
\end{theorem}
Using results of Harman~\cites{OnTheNumberOfCarmichael,WattsMean} one might be able to improve the factor $\frac{5}{12}$ in the exponent to $0.4736$.

\section{Notation}

As usual $\mathbb{N}$ denotes the set of positive integers and $\mathbb{P}$ the set of primes in $\mathbb{N}$. We denote the greatest common divisor and the least common multiple of $n$ elements $a_i \in \mathbb{N}$ by $\gcd(a_1,a_2, \ldots, a_n)$ and $\lcm(a_1,a_2, \ldots, a_n)$ or $(a_1, a_2, \ldots, a_n)$ and $[a_1, a_2, \ldots, a_n]$ for short. For integers $d,n \in \mathbb{N}$ we write $d|n$ if $d$ divides $n$. We use the symbols $\mathcal{O}$, $o$, $\ll$ and $\gg$ within the contexts of the well known Landau and Vinogradov notations where dependence of the implied constant on certain variables is indicated by a subscript. For any prime $p \in \mathbb{P}$ we define the function $\nu_p:\mathbb{N}\rightarrow \mathbb{N}\cup \{0\}$ to be the $p$-adic valuation, i.e. $\nu_p(n)=a$ if and only if $p^a$ is the highest power of $p$ dividing $n$. By $\tau(n)$ and $\omega(n)$, as usual, we denote the number of divisors and the number of distinct prime divisors of $n$. By $\tau(n,m)$, we denote the number of divisors of $n$ coprime to $m$ and $\tau(n,k,m)$, $\omega(n,k,m)$ denote the number of divisors (resp. distinct prime divisors) of $n$ in the residue class $k \bmod m$, where $(k,m)=1$. Finally, for two coprime integers $a$ and $b$ we denote by $\ord_a(b)$ the least positive integer $l$, such that $b^l\equiv 1 \bmod a$. 

\section{Heuristics on $f_k(m,n)$} \label{sec: heuristics}

We now informally discuss why $f_3(m,n)=\mathcal{O}_{\epsilon}(n^{\epsilon})$ can be expected. In fact, as far as we are aware, this was first observed by Roger Heath-Brown (private communication with the first author in 1994). Let us first recall (see e.g.~\cite{TheoryOf}*{p. 201: Theorem 3}) that a fraction $\frac{m}{n}$ with $\gcd(m,n)=1$ is a sum of two unit fractions $\frac{1}{a_1}+\frac{1}{a_2}$ if and only if there exist two distinct, positive and coprime divisors $d_1$ and $d_2$ of $n$ such that $d_1+d_2\equiv 0 \bmod m$. We may deduce an upper bound of $\mathcal{O}_{\epsilon}(n^{\epsilon})$ for the number of representations of $\frac{m}{n}$ as a sum of two unit fractions. Indeed from 
\begin{equation} \label{eq: heuristics two fractions equation}
\frac{m}{n}=\frac{1}{a_1}+\frac{1}{a_2},
\end{equation} 
by setting $d=(a_1,a_2)$ and $a_i^{\prime}=\frac{a_i}{d}$ for $i\in\{1,2\}$, we see that
$$ma_1^{\prime}a_2^{\prime}d=n(a_1^{\prime}+a_2^{\prime}).$$
This implies that $a_1^{\prime},a_2^{\prime}$ are divisors of $n$, $d$ divides $n(a_1^{\prime}+a_2^{\prime})<2n^2$ and any solution $(a_1,a_2)$ of \eqref{eq: heuristics two fractions equation} uniquely corresponds to a triple $(a_1^{\prime},a_2^{\prime},d)$. The number $\sum_{a_1^{\prime},a_2^{\prime}|n}\tau(n(a_1^{\prime}+a_2^{\prime}))$ of such triples is bounded by $\mathcal{O}_{\epsilon}(n^{\epsilon})$ (see Lemma~\ref{lem: divisor bound} below).
 
Studying $\frac{m}{n}=\frac{1}{a_1}+\frac{1}{a_2}+\frac{1}{a_3}$ with $a_1\leq a_2 \leq a_3$ one observes that
$$\frac{1}{a_1} <\frac{m}{n}\leq \frac{3}{a_1}$$
from which $\frac{n}{m}<a_1\leq \frac{3n}{m}$ follows. In view of
\begin{equation} \label{eq: heuristic sum of three unit fractions}
\frac{m}{n}-\frac{1}{a_1}=\frac{ma_1-n}{na_1}=\frac{1}{a_2}+\frac{1}{a_3}
\end{equation}
there are at most $\mathcal{O}\left(\frac{n}{m}\right)$ choices for $a_1$, and for given $a_1$ there are at most $d(na_1)=\mathcal{O}_{\epsilon}(n^{\epsilon})$ divisors of $na_1$. This shows that $f_3(m,n)=\mathcal{O}_{\epsilon}\left(\frac{n^{1+\epsilon}}{m}\right)$ is a trivial upper bound. The real question is for how many values of $a_1$ there can be at least one solution. For increasing $a_1$, even if $na_1$ contains many divisors, the congruence $d_1+d_2 \equiv 0 \bmod ma_1-n$ should become, on average, more difficult to satisfy if $ma_1-n\gg n^{\epsilon}$. Therefore we expect that the number of $a_1$ contributing at least one solution is $\mathcal{O}_{\epsilon}(n^{\epsilon})$, so that $f_3(m,n)=\mathcal{O}_{\epsilon}(n^{2\epsilon})$. Moreover equation~\eqref{eq: heuristic sum of three unit fractions} implies that for any given $a_1$, the number of solutions is about $\tilde{d}(m,n,a_1)$. Here $\tilde{d}(m,n,a_1)$ counts the number of pairs of coprime divisors $d_1,d_2$ of $na_1$, with $d_1+d_2\equiv 0 \bmod ma_1-n$. Therefore $f_3(m,n)$ should be approximately $\sum_{a_1}\tilde{d}(m,n,a_1)$.

Similarly a completely trivial upper bound on $f_4(m,n)$ is as follows. With $a_1\leq a_2\leq a_3\leq a_4$ it follows that $\frac{n}{m}<a_1\leq \frac{4n}{m}$ and hence
$$\frac{ma_1-n}{na_1}=\frac{m}{n}-\frac{1}{a_1}=\frac{1}{a_2}+\frac{1}{a_3}+\frac{1}{a_4}\leq \frac{3}{a_2}.$$
From those bounds we easily deduce that $a_2 \leq \frac{12n^2}{m}$. With 
$$\frac{m}{n}-\frac{1}{a_1}-\frac{1}{a_2}=\frac{ma_1a_2-na_2-na_1}{na_1a_2}=\frac{1}{a_3}+\frac{1}{a_3},$$
with similar arguments as above, we deduce that $f_4(m,n)=\mathcal{O}_{\epsilon}\left(\frac{n^{3+\epsilon}}{m^2}\right)$. For fixed $m$ the fact that our bound on $f_4(m,n)$ in Theorem~\ref{thm: k fractions theorem} below is better than $\mathcal{O}(n^2)$ shows that, for most pairs $(a_1,a_2)$ and moreover, for most choices of $a_2 \in \left[\frac{n}{m},\frac{12n^2}{m}\right]$ there is no solution of $\frac{m}{n}=\frac{1}{a_1}+\frac{1}{a_2}+\frac{1}{a_3}+\frac{1}{a_4}$. Here again, as soon as $ma_1a_2-na_2-na_1 \gg n^{\epsilon}$ one should not expect to have two divisors $d_1,d_2$ of $na_1a_2$ such that $d_1+d_2\equiv 0 \bmod ma_1a_2-na_2-na_1$. From this reasoning, also $f_k(m,n)=\mathcal{O}_{\epsilon,k}(n^{\epsilon})$, for $k \geq 4$ seems to us a reasonable expectation.

The papers~\cite{TheNumber} and~\cite{CountingThe} studied parametric solutions of the diophantine equation~\eqref{eq: 3 fractions main equation}. The reason why the result in~\cite{CountingThe} is superior in the case of $n$ being a prime is that here a full parametric solution (e.g.~\cite{SullEquazione}) is much easier to work with. However, in this manuscript we develop parametric solutions of~\eqref{eq: 3 fractions main equation} and~\eqref{eq: k fractions main equation} from scratch. Some simplified version of this has been used in~\cite{SumsOfArticle} and~\cite{CountingThe}*{Section 11}, but there the focus was to generate solutions with many parameters. Here we need to do kind of the opposite, namely to show that \textit{every} solution comes from a number of parametric families. 

The method we introduce should theoretically work for any diophantine equation as it expresses a $k$-tuple of integers in a standard form. In practice it might work favorably if there is some inhomogeneous part as in
$$n=a_1a_2a_3-a_1-a_2.$$
For prime values of $n$ in equation~\eqref{eq: 3 fractions main equation} there are several discussions of parametric solutions in the literature, e.g. by Rosati~\cite{SullEquazione} and Aigner~\cite{BruecheAls}, see also Mordell's book~\cite{DiophantineEquations}*{Chapter 30}. For composite values $n$ there is no satisfactory treatment in the literature, and Section~\ref{sec: 3 unit fractions} below may be the most detailed study to date.

\section{Patterns and relative greatest common divisors} \label{sec: patterns and rgcds}

Consider a solution $(a_1,a_2, \ldots, a_k) \in \mathbb{N}^k$ with $a_1 \leq a_2 \leq \ldots \leq a_k$ of equation \eqref{eq: k fractions main equation} and set $n_i=(a_i,n)$, $a_i=n_it_i$ for $i \in \{1,2, \ldots, k\}$. We can thus rewrite equation \eqref{eq: k fractions main equation} as
\begin{equation} \label{eq: k fractions pattern equation}
\frac{m}{n}=\sum_{i=1}^k{\frac{1}{n_it_i}}.
\end{equation}
Later, when working on upper bounds for the number of solutions of equation \eqref{eq: k fractions pattern equation} for $k \in \{3,4\}$, we will fix a choice of $(n_1,n_2, \ldots, n_k) \in \mathbb{N}^k$. For given $m,n \in \mathbb{N}$ we call such a choice the \emph{pattern} of a solution of this equation. Note that for solutions corresponding to a given pattern $(n_1, n_2, \ldots, n_k)$ we have that $\left(\frac{n}{n_i},t_i\right)=1$ for all $i \in \{1,2, \ldots, k\}$. As $n_i|n$ the number of distinct patterns is $\mathcal{O}_k(n^{\epsilon})$ only.

Also, when dealing with equations of type \eqref{eq: k fractions pattern equation} for $k \in \{3,4\}$ we will make heavy use of the concept of relative greatest common divisors as described by Elsholtz in~\cite{SumsOf} (for some ad hoc definition see also~\cite{SumsOfArticle}). Relative greatest common divisors are a useful tool when studying divisibility relations among the $t_i$ in \eqref{eq: k fractions pattern equation}.

Let $I=\{1,2, \ldots, k\}$ be the index set. Then we define the relative greatest common divisors of the positive integers $t_1,t_2, \ldots, t_k$ recursively as follows:
$$x_I=\gcd(t_1,t_2, \ldots, t_k)$$
and for any $\{i_1,i_2,\ldots i_{|J|}\}=J \subseteq I$, $J \neq \emptyset$ we set
$$x_J=\frac{\gcd(t_{i_1},t_{i_2}, \ldots, t_{i_{|J|}})}{\prod_{\substack{J^{\prime} \subseteq I \\ J \subsetneq J^{\prime}}}x_{J^{\prime}}}.$$
For $k \in \{3,4\}$ we will later identify the elements $x_J$ with $J \subseteq I$ with the elements $x_i, x_{ij}$ and $x_{ijk}$ where $\{i,j,k\}=\{1,2,3\}$ in the case when $k=3$ and with the elements $x_{i}, x_{ij},x_{ijk}$ and $x_{ijkl}$ with $\{i,j,k,l\}=\{1,2,3,4\}$ when $k=4$. With the relative greatest common divisors defined as above we have that
$$t_i=\prod_{\substack{J \subseteq I \\ i\in J}}x_J.$$

A further very useful property of relative greatest common divisors is that $(x_J,x_K)=1$ if $J \nsubseteq K$ and $K \nsubseteq J$. We prove this property as the following lemma (see also~\cite{SumsOf}*{p. 2}).
\begin{lemma} \label{lem: rgcd lemma}
Let $t_1, t_2, \ldots, t_k \in \mathbb{N}$, $J,K \subseteq \{1,2, \ldots, k\}$, $J,K \neq \emptyset$ and define the corresponding relative greatest common divisors $x_J$ and $x_K$ as above. If $J \nsubseteq K$ and $K \nsubseteq J$ then $(x_J,x_K)=1$.
\end{lemma}
\begin{proof}
By assumption $J \nsubseteq K$ and $K \nsubseteq J$ and thus we have that $J \subsetneq J \cup K$ and $K \subsetneq J \cup K$. We suppose that $d=(x_J,x_K)>1$ and choose an arbitrary prime divisor $p|d$. Set $L=J \cup K$, $J=\{j_1,j_2, \ldots, j_{|J|}\}$, $K=\{k_1,k_2, \ldots, k_{|K|}\}$, $L=\{l_1,l_2, \ldots, l_{|L|}\}$ and write
\begin{align*}
x_J&=\frac{(t_{j_1},t_{j_2}, \ldots, t_{j_{|J|}})}{\left(\prod_{\substack{J^{\prime}\subseteq{I} \\ J \subsetneq J^{\prime} \\ L \nsubseteq J^{\prime}}}x_{J^{\prime}}\right)\cdot x_L \cdot \left(\prod_{\substack{J^{\prime} \subseteq I \\ L \subsetneq J^{\prime}}}x_{J^{\prime}}\right)} \text{, } \\
x_K&=\frac{(t_{k_1},t_{k_2}, \ldots, t_{k_{|K|}})}{\left(\prod_{\substack{K^{\prime}\subseteq{I} \\ K \subsetneq K^{\prime} \\ L \nsubseteq K^{\prime}}}x_{K^{\prime}}\right)\cdot x_L \cdot \left(\prod_{\substack{K^{\prime} \subseteq I \\ L \subsetneq K^{\prime}}}x_{K^{\prime}}\right)}.
\end{align*}
With $x_L=\frac{(t_{l_1},t_{l_2}, \ldots, t_{l_{|L|}})}{\prod_{\substack{L^{\prime}\subseteq I \\ L \subsetneq L^{\prime}}}x_{L^{\prime}}}$ this simplifies to
\begin{equation} \label{eq: rgcd lemma equation}
x_J=\frac{(t_{j_1},t_{j_2}, \ldots, t_{j_{|J|}})}{\left(\prod_{\substack{J^{\prime}\subseteq{I} \\ J \subsetneq J^{\prime} \\ L \nsubseteq J^{\prime}}}x_{J^{\prime}}\right)\cdot (t_{l_1},t_{l_2}, \ldots, t_{l_{|L|}})} \text{, } x_K=\frac{(t_{k_1},t_{k_2}, \ldots, t_{k_{|K|}})}{\left(\prod_{\substack{K^{\prime}\subseteq{I} \\ K \subsetneq K^{\prime} \\ L \nsubseteq K^{\prime}}}x_{K^{\prime}}\right)\cdot (t_{l_1},t_{l_2}, \ldots, t_{l_{|L|}})}.
\end{equation}
Let $p^{\alpha}$ be the highest power of $p$ dividing the greatest common divisor of the terms $(t_{j_1},t_{j_2}, \ldots, t_{j_{|J|}})$ and $(t_{k_1}, t_{k_1}, \ldots, t_{k_{|K|}})$. Thus $p^{\alpha}$ is also the highest power of $p$ such that
$$p^{\alpha} | ((t_{j_1},t_{j_2}, \ldots, t_{j_{|J|}}),(t_{k_1}, t_{k_1}, \ldots, t_{k_{|K|}}))=(t_{l_1},t_{l_2}, \ldots, t_{l_{|L|}}).$$
By definition of the greatest common divisor, without loss of generality we may suppose that $\nu_p((t_{j_1},t_{j_2}, \ldots, t_{j_{|J|}}))=\alpha$. From equation \eqref{eq: rgcd lemma equation} we finally see that $\nu_p(x_J)=0$, a contradiction to $p|d$.
\end{proof}

Relative greatest common divisors may be nicely visualized via Venn diagrams (especially when $k \leq 3$). We identify a positive integers with the multiset of its prime divisors, i.e. each prime $p$ dividing $n$ occurs with multiplicity $\nu_p(n)$ in the multiset. Given the Venn diagram of the multisets corresponding to the integers $t_1, \ldots ,t_k$, each area of intersection in the diagram uniquely corresponds to a relative greatest common divisor $x_J$, $J \subseteq \{1,\ldots,k\}$. Figure~\ref{fig: Venn Diagram} shows the situation for relative greatest common divisors of three positive integers $t_1,t_2$ and $t_3$.
\begin{figure}[ht]
\includegraphics[scale=10]{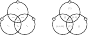}
\caption{A visualization of relative greatest common divisors using Venn diagrams. On the left hand side one sees the general case of three positive integers $t_1,t_2$ and $t_3$ and on the right hand side the situation when $t_1=90$, $t_2=126$ and $t_3=616$. Empty sets correspond to empty products and we set the corresponding relative greatest common divisor to $1$.}
\label{fig: Venn Diagram}
\end{figure} 

As mentioned in the beginning of this section relative greatest common divisors were systematically described in~\cite{SumsOf}. Nonetheless concepts of a similar type date back at least as far as Dedekind~\cite{UeberZerlegungen} who called the relative greatest common divisors of the integers $t_1, \ldots, t_k$ the cores (Kerne) of the system $(t_1, \ldots, t_k)$. Dedekind described the construction of these cores explicitly for systems with three and four elements and developed some theory to describe the cores of systems with more than four elements.

Decompositions similar to relative greatest common divisors also occur when we look for generalizations of the formula
\begin{equation} \label{eq: lcm gcd equation}
[t_1,t_2]=\frac{t_1t_2}{(t_1,t_2)},
\end{equation}
where $[t_1,t_2]$ denotes the least common multiple of the integers $t_1$ and $t_2$. A generalization of formula~\eqref{eq: lcm gcd equation} to least common multiples and greatest common divisors of $k$ integers $t_1, \ldots ,t_k$ was found by V.-A. Lebesgue~\cite{ExtraitDes}*{p. 350}, who proved that
$$[t_1,t_2, \ldots, t_k]=\frac{\prod_{\substack{1 \leq i \leq k \\ i \text { odd}}}G_i}{\prod_{\substack{1 \leq j \leq k \\ j \text { even}}}G_j},$$
where the variables $G_i$ denote the product of the greatest common divisors of all choices of subsets of $i$ integers in the set $\{t_1, t_2, \ldots, t_k\}.$

\section{Sums of three unit fractions} \label{sec: 3 unit fractions}

In this section we deal with equation \eqref{eq: k fractions pattern equation} for $k=3$, i.e. with equations of the form
\begin{equation} \label{eq: general Erdoes Straus equation}
\frac{m}{n}=\frac{1}{n_1t_1}+\frac{1}{n_2t_2}+\frac{1}{n_3t_3},
\end{equation}
where $n_1t_1\leq n_2t_2\leq n_3t_3$, $n_i|n$ and $\left(\frac{n}{n_i},t_i\right)=1$ for $i \in \{1, 2, 3\}$. In the following we use the concept of relative greatest common divisors introduced in the previous section to get a suitable parametrisation of the solutions of \eqref{eq: general Erdoes Straus equation} corresponding to a fixed pattern $(n_1,n_2,n_3) \in \mathbb{N}^3$.  

Writing the variables $t_i$ in terms of relative greatest common divisors, equation \eqref{eq: general Erdoes Straus equation} takes the form
\begin{equation} \label{eq: xi elimination equation before multiplying out}
\frac{m}{n}=\frac{1}{n_1x_1x_{12}x_{13}x_{123}}+\frac{1}{n_2x_2x_{12}x_{23}x_{123}}+\frac{1}{n_3x_3x_{13}x_{23}x_{123}}
\end{equation}
and multiplying out yields
\begin{equation} \label{eq: xi elimination equation}
mx_1x_2x_3x_{12}x_{13}x_{23}x_{123}=\frac{n}{n_1}x_2x_3x_{23}+\frac{n}{n_2}x_1x_3x_{13}+\frac{n}{n_3}x_1x_2x_{12}.
\end{equation}
A first thing we observe is that we have $x_i=1$ for all $i \in \{1,2,3\}$. This follows from Lemma~\ref{lem: rgcd lemma} and equation \eqref{eq: xi elimination equation} together with the fact that $x_i|\frac{n}{n_i}$ is possible only if $x_i=1$ by definition of $n_i$. We thus can work with the following simplified version of equation \eqref{eq: xi elimination equation}
\begin{equation} \label{eq: final multiplied out equation}
mx_{12}x_{13}x_{23}x_{123}=\frac{n}{n_1}x_{23}+\frac{n}{n_2}x_{13}+\frac{n}{n_3}x_{12}.
\end{equation}
Next we introduce the parameters $d_{ij}$ which are defined as $d_{ij}=\left(\frac{n}{n_i},\frac{n}{n_j}\right)$. Again we have that $(x_{ij},d_{ij})=1$ by definition of the $n_i$ and we note that for given $m,n$ and a fixed pattern $(n_1,n_2,n_3)$ also the parameters $d_{ij}$ are fixed.

In what follows we apply methods developed by Elsholtz and Tao~\cite{CountingThe}*{Sections 2 and 3}. The strategy is to derive a system of equations from \eqref{eq: final multiplied out equation} and to make use of divisor relations therein. With the observation of coprimality of $d_{ij}$ and $x_{ij}$, and using divisibility relations implied by equation \eqref{eq: final multiplied out equation} we may define the following three positive integers
$$w=\frac{\frac{n}{n_1d_{13}}x_{23}+\frac{n}{n_3d_{13}}x_{12}}{x_{13}} \text{, }y=\frac{\frac{n}{n_1d_{12}}x_{23}+\frac{n}{n_2d_{12}}x_{13}}{x_{12}} \text{ and } z=\frac{\frac{n}{n_2d_{23}}x_{13}+\frac{n}{n_3d_{23}}x_{12}}{x_{23}}.$$
Later we make use of the product of $w$ and $z$ which is given by
\begin{align*}
wz&=\frac{n}{n_1d_{13}}\frac{n}{n_2d_{23}}+\frac{x_{12}}{x_{13}x_{23}}\left(\frac{n^2}{n_1n_3d_{13}d_{23}}x_{23}+\frac{n^2}{n_2n_3d_{13}d_{23}}x_{13}+\frac{n^2}{n_3^2d_{13}d_{23}}x_{12}\right) \\
&=\frac{n}{n_1d_{13}}\frac{n}{n_2d_{23}} + \frac{nx_{12}}{n_3d_{13}d_{23}x_{13}x_{23}}\left(\frac{n}{n_1}x_{23}+\frac{n}{n_2}x_{13}+\frac{n}{n_3}x_{12}\right) \\
&=\frac{n}{n_1d_{13}}\frac{n}{n_2d_{23}} + \frac{nm}{n_3d_{13}d_{23}}x_{12}^2x_{123},
\end{align*}
where we used equation \eqref{eq: final multiplied out equation} to get the last equality. We collect the equations just derived in the following list
\begin{align}
mx_{12}x_{13}x_{23}x_{123}&=\frac{n}{n_1}x_{23}+\frac{n}{n_2}x_{13}+\frac{n}{n_3}x_{12} \label{eq: stystem equation 1} \\
yx_{12}&=\frac{n}{n_1d_{12}}x_{23}+\frac{n}{n_2d_{12}}x_{13} \label{eq: system equation 2} \\
zx_{23}&=\frac{n}{n_2d_{23}}x_{13}+\frac{n}{n_3d_{23}}x_{12} \label{eq: system equation 3}\\
mx_{13}x_{23}x_{123}&=d_{12}y+\frac{n}{n_3} \label{eq: system equation 4}\\
mx_{12}x_{13}x_{123}&=d_{23}z+\frac{n}{n_1} \label{eq: system equation 5}\\
wz&=\frac{n}{n_1d_{13}}\frac{n}{n_2d_{23}} + \frac{nm}{n_3d_{13}d_{23}}x_{12}^2x_{123}. \label{eq: system equation 6}
\end{align}
For proving Theorem~\ref{thm: main theorem} the classical divisor bound will play a crucial role. We will use it in the following form (see \cite{AnIntroduction}*{Theorem 315}).

\begin{customlem}A \label{lem: divisor bound}
Let $d(n):\mathbb{N}\rightarrow \mathbb{N}$ be the divisor function, i.e. $d(n)=\sum_{d|n}1$. Then for every $\epsilon >0$ we have
$$d(n)\ll_{\epsilon} n^{\epsilon}.$$
\end{customlem}
We now have all the tools we need to prove Theorem~\ref{thm: main theorem}.

\begin{proof}[Proof of Theorem~\ref{thm: main theorem}]
Consider a solution of equation \eqref{eq: general Erdoes Straus equation} for a fixed pattern $(n_1,n_2,n_3)$. By assumption we have $n_1t_1 \leq n_2t_2 \leq n_3t_3$ and using the parametrization of the $t_i$ we introduced in equation \eqref{eq: xi elimination equation before multiplying out} this implies
$$x_{13} \leq \frac{n_2}{n_1}x_{23} \text{ and } x_{12} \leq \frac{n_3}{n_2}x_{13}.$$
Using these inequalities in equations \eqref{eq: system equation 2} and \eqref{eq: system equation 3} yields
$$yx_{12} \leq 2\frac{n}{n_1d_{12}}x_{23} \text{ and } zx_{23} \leq 2\frac{n}{n_2d_{23}}x_{13}.$$
Dividing by $x_{23}$ and $x_{13}$ respectively and multiplying the last two inequalities we arrive at
$$\frac{yx_{12}}{x_{23}}\frac{zx_{23}}{x_{13}} \leq 4\frac{n^2}{n_1n_2d_{12}d_{23}}.$$
We now intend to obtain a lower bound for $n_1n_2d_{12}d_{23}$. Let $n=\prod_{p \in \mathbb{P}}p^{\nu_p(n)}$ be the prime factorization of $n$. Then $n_1=\prod_{p \in \mathbb{P}}p^{\nu_p(n_1)}$ and $n_2=\prod_{p \in \mathbb{P}}p^{\nu_p(n_2)}$ where $0 \leq \nu_p(n_1),\nu_p(n_2) \leq \nu_p(n)$ for all $p \in \mathbb{P}$. Since 
$$d_{12} =\left(\frac{n}{n_1},\frac{n}{n_2}\right) =\prod_{p \in \mathbb{P}}p^{\nu_p(n)-\max(\nu_p(n_1),\nu_p(n_2))}$$ 
we have
\begin{align*}
n_1n_2d_{12} &= \prod_{p \in \mathbb{P}}p^{\nu_p(n_1)+\nu_p(n_2)+\nu_p(n)-\max(\nu_p(n_1),\nu_p(n_2))} \\
&\geq \prod_{p \in \mathbb{P}}p^{\nu_p(n_1)+\nu_p(n_2)+\nu_p(n)-\nu_p(n_1)-\nu_p(n_2)}=n.
\end{align*}
This shows that $n_1n_2d_{12}d_{23} \geq n$ and thus
$$\frac{yx_{12}}{x_{23}}\frac{zx_{23}}{x_{13}} \ll n.$$
By assumption we have that $n_1t_1$ is the smallest denominator in equation \eqref{eq: general Erdoes Straus equation}. This implies that 
$$\frac{m}{n} \leq \frac{3}{n_1t_1} \text{ and thus } t_1 \leq \frac{3n}{mn_1} \ll \frac{n}{m}.$$
The bound in Theorem~\ref{thm: main theorem} can finally be derived from the following inequality
\begin{equation} \label{eq: 1 5 equation}
y \cdot z \cdot x_{12}x_{13} \cdot (x_{12}x_{123})^2=\frac{yx_{12}}{x_{23}}\frac{zx_{23}}{x_{13}}(x_{12}x_{13}x_{123})^2 \ll \frac{n^3}{m^2}.
\end{equation}
This implies that at least one of the factors $y$, $z$, $x_{12}x_{13}$ and $x_{12}x_{123}$ is bounded by $\mathcal{O}\left(\left(\frac{n^3}{m^2}\right)^{\nicefrac{1}{5}}\right)$.

If this is the case for $y$ then by Lemma~\ref{lem: divisor bound} and equation \eqref{eq: system equation 4} we have at most $\mathcal{O}_{\epsilon}(n^{\epsilon})$ choices for the parameters $x_{13}$, $x_{23}$ and $x_{123}$ for every choice of $y$. The parameter $x_{12}$ is then uniquely determined by \eqref{eq: stystem equation 1}.

Similarly, if $z$ is the bounded parameter use Lemma~\ref{lem: divisor bound} and equation \eqref{eq: system equation 5} to see that there are at most $\mathcal{O}_{\epsilon}(n^{\epsilon})$ choices for the parameters $x_{12}$, $x_{13}$ and $x_{123}$ for every choice of $z$. Again the remaining parameter $x_{23}$ is uniquely determined by \eqref{eq: stystem equation 1}.

Suppose that $x_{12}x_{13} \ll \left(\frac{n^3}{m^2}\right)^{\nicefrac{1}{5}}$. By Lemma~\ref{lem: divisor bound} for every fixed choice of $x_{12}x_{13}$ we may choose the factors $x_{12}$ and $x_{13}$ in at most $\mathcal{O}_{\epsilon}(n^{\epsilon})$ ways. For each of those choices Lemma~\ref{lem: divisor bound} and equation \eqref{eq: system equation 3} imply that there are at most $\mathcal{O}_{\epsilon}(n^{\epsilon})$ choices for the parameter $x_{23}$. As before the remaining parameter $x_{123}$ is then fixed by \eqref{eq: stystem equation 1}.

Finally we need to consider the case when $x_{12}x_{123}$ is the bounded factor. As in the previous case for any fixed choice of $x_{12}x_{123}$ we have at most $\mathcal{O}_{\epsilon}(n^{\epsilon})$ choices for the factors $x_{12}$ and $x_{123}$. Since equation \eqref{eq: general Erdoes Straus equation} has no solutions for $m > 3n$ we have that $m \ll n$ and using equation \eqref{eq: system equation 6} we see that for any fixed choice of $x_{12}$ and $x_{123}$ we have at most $\mathcal{O}_{\epsilon}(n^{\epsilon})$ choices for the parameters $w$ and $z$. With $z$, $x_{12}$ and $x_{123}$ fixed, $x_{13}$ is uniquely determined by \eqref{eq: system equation 5}. The last parameter $x_{23}$ is again uniquely determined by \eqref{eq: stystem equation 1}.

In any case we have a bounded number of applications of the divisor bound from Lemma~\ref{lem: divisor bound}, say it was applied at most $l$ times. Setting $\tilde{\epsilon}=l\epsilon$ we hence have at most $\mathcal{O}_{\tilde{\epsilon}}\left(n^{\tilde{\epsilon}}\left(\frac{n^3}{m^2}\right)^{\nicefrac{1}{5}}\right)$ choices for the parameters $x_{12}$, $x_{13}$, $x_{23}$ and $x_{123}$ which uniquely determine a solution of \eqref{eq: general Erdoes Straus equation} if $n_1$, $n_2$ and $n_3$ are fixed. Note that this bound is independent of the concrete choice of the parameters $n_i$ and again by Lemma~\ref{lem: divisor bound} we have at most $\mathcal{O}_{\epsilon}(n^{3\epsilon})$ choices for the pattern $(n_1,n_2,n_3)$. Theorem~\ref{thm: main theorem} now follows by redefining the choice of $\epsilon$. 
\end{proof}

Finally we prove Corollary~\ref{cor: algorithmic aspects}.

\begin{proof}[Proof of Corollary~\ref{cor: algorithmic aspects}.]
The proof of Theorem~\ref{thm: main theorem} suggests an algorithm for computing all decompositions of a rational number $\frac{m}{n}$ as a sum of three unit fractions. The running time of this algorithm depends on the quality of algorithms used for integer factorization. In~\cite{ARigorous} a probabilistic algorithm is analyzed which finds all prime factors of a given integer in expected running time $\exp((1+o(1))\sqrt{\log n \log \log n})$ for $n \rightarrow \infty$, which is clearly $\mathcal{O}_{\epsilon}(n^{\epsilon})$. Here the term probabilistic means that the algorithm is allowed to call a random number generator which outputs $0$ or $1$ each with probability $\frac{1}{2}$. The term expected running time refers to averaging over the output of the random number generator only and not over the input $n$. Hence the expected running time is also valid for each individual $n$.

As a consequence, using an algorithm of this type, all decompositions of $\frac{m}{n}$ as a sum of three unit fractions can be found by carrying out the following steps. Factorize the integer $n$ and compute all possible patterns $(n_1,n_2,n_3)$. For any of these $\mathcal{O}_{\epsilon}(n^{\epsilon})$ patterns it follows from the calculations in the proof of Theorem~\ref{thm: main theorem}, that the implied constant in inequality~\eqref{eq: 1 5 equation} may be chosen as $C:=\left(\frac{36}{n_1^2d_{23}}\right)$. For all choices of integers $y$, $z$, $x_{12}x_{13}$ and $x_{12}x_{123} \in \left[1, C^{\nicefrac{1}{5}}\left(\frac{n^3}{m^2}\right)^{\nicefrac{1}{5}}\right]$ we determine the integers $x_{12},x_{13},x_{23}$ and $x_{123}$ via factoring $x_{12}x_{13}$, $x_{12}x_{123}$ and a small number of integers mentioned in formulae \eqref{eq: stystem equation 1}-\eqref{eq: system equation 6}. All in all this leads to an algorithm of expected running time $\mathcal{O}_{\epsilon}\left(n^{\epsilon}\left(\frac{n^3}{m^2}\right)^{\nicefrac{1}{5}}\right)$.

As for representations of the form
\begin{equation} \label{eq: k fractions equation algorithmic aspects}
\frac{m}{n}=\sum_{i=1}^k\frac{1}{a_i}
\end{equation}
with $k>3$ we enumerate all possible choices for the denominators $a_i$, $1 \leq i \leq k-3$, and apply our algorithm for finding representations as sum of three unit fractions to determine all choices for the remaining three denominators, i.e. we solve 
\begin{equation} \label{eq: k fractions last three denominators excluded}
\frac{m}{n}-\sum_{i=1}^{k-3}\frac{1}{a_i}=\frac{1}{a_{k-2}}+\frac{1}{a_{k-1}}+\frac{1}{a_k}.
\end{equation}
We suppose the denominators $a_i$ in equation \eqref{eq: k fractions equation algorithmic aspects} are given in increasing order and prove upper bounds for the size of $a_i$, $1\leq i\leq k$. In particular we use an induction argument to show that $a_i \leq \alpha_i n^{2^{i-1}}$ where the finite sequence $(\alpha_i)_{1 \leq i \leq k}$ is recursively defined by $\alpha_1=k$ and $\alpha_i=(k-i+1)\prod_{j<i}\alpha_j$ for $2 \leq i \leq k$. For $i=1$ this bound follows easily from the following inequality
$$\frac{m}{n}=\frac{1}{a_1}+\cdots+\frac{1}{a_k}\leq \frac{k}{a_1}$$
which leads to $a_1 \leq \frac{kn}{m}\leq kn$. If we suppose the bound holds for $a_i$, with a similar argument we get
$$\frac{m}{n}-\frac{1}{a_1}-\cdots-\frac{1}{a_i}=\frac{1}{a_{i+1}}+\cdots+\frac{1}{a_k}\leq \frac{(k-i)}{a_{i+1}}.$$
The last inequality together with the induction hypothesis for $j<i+1$ implies 
$$a_{i+1}\leq (k-i)\frac{n\prod_{j<i+1}a_j}{m\prod_{j<i+1}a_j-n\sum_{j<i+1}\prod_{\substack{l <i+1 \\l\neq j }}a_l}\leq (k-i)n\prod_{j<i+1}a_j\leq \alpha_{i+1}n^{2^i}.$$
By definition $\alpha_i$ is a polynomial in $k$ of degree $2^i$ with leading coefficient $1$. Furthermore the denominator of the rational number on the left hand side of equation \eqref{eq: k fractions last three denominators excluded} is of size at most $n\prod_{i=1}^{k-3}a_i\ll_k n^{2^{k-3}}$. By the aforementioned result we can compute all decompositions as a sum of three unit fractions of this number in time $\mathcal{O}_{\epsilon,k}(n^{2^{k-3}(\nicefrac{3}{5} +\epsilon)})$. We have to compute these representations for at most $\prod_{i=1}^{k-3}a_i \ll_kn^{2^{k-3}-1}$ rational numbers which leads to an upper bound of
$$\mathcal{O}_{\epsilon,k}\left(n^{2^{k-3}(\nicefrac{8}{5}+\epsilon)-1}\right)$$
for the running time.
\end{proof}

\begin{remark}
The procedure for computing representations as a sum of $k$ unit fractions as described in the proof of Corollary~\ref{cor: algorithmic aspects} could lead to a speedup for calculations similar to those in~\cite{OnTheNumber}. In the calculations above the size of the numerator of the rational number on the left hand side of equation~\eqref{eq: k fractions last three denominators excluded}, which we denote by $\frac{m^{\prime}}{n^{\prime}}$, was not taken into account. We note that also the proof of the upper bound for $f_3(m,n)$ by Browning and Elsholtz~\cite{TheNumber}*{Theorem 2} may be similarly turned into an algorithm of running time $\mathcal{O}_{\epsilon}\left(n^{\epsilon}\left(\frac{n}{m}\right)^{\nicefrac{2}{3}}\right)$. In practice one would check dynamically if $m^{\prime} \ll (n^{\prime})^{\nicefrac{1}{4}}$ before computing the representations as a sum of three unit fractions of $\frac{m^{\prime}}{n^{\prime}}$. If this is the case, the algorithm described in the first part of the proof of Corollary~\ref{cor: algorithmic aspects} should be applied, if $m^{\prime} \gg (n^{\prime})^{\nicefrac{1}{4}}$ the method of~\cite{TheNumber} should be used.  
\end{remark}

\section{Sums of $k$ unit fractions} 

In this section we will prove Theorem~\ref{thm: k fractions theorem}. Browning and Elsholtz used an induction argument on their bound for the quantity $f_3(m,n)$ to get bounds for $f_k(m,n)$ for $k \geq 4$. Using their arguments directly on our result from Theorem~\ref{thm: main theorem} would lead to worse upper bounds than those of Browning and Elsholtz. The reason is that our bound for $f_3(m,n)$ is weaker than the one in~\cite{TheNumber} when $m$ is large.

As in~\cite{TheNumber}*{Section 4} the proof of Theorem~\ref{thm: k fractions theorem} will be based on the observation that from equation \eqref{eq: k fractions pattern equation} it follows that 
$$f_k(m,n)\leq \sum_{\frac{n}{m} < n_1t_1 \leq \frac{kn}{m} }f_{k-1}(mn_1t_1-n,n_1t_1n),$$
which, after introducing the parameter $u=mn_1t_1-n$, becomes
\begin{equation} \label{eq: Browning Elsholtz u equation}
f_k(m,n)\leq \sum_{\substack{0 < u \leq (k-1)n \\m|u+n}}f_{k-1}\left(u,\frac{n(u+n)}{m}\right).
\end{equation}
 
The improvement in Theorem~\ref{thm: k fractions theorem} stems from extending the method of Browning and Elsholtz by applying the following new idea. In the case of $k=4$ we do not consider the sum on the right hand side of \eqref{eq: Browning Elsholtz u equation} as a whole but we split the sum into two parts. In the first part we collect the values of $u$ where $0<u\leq n^{\delta}$ for some $0<\delta <1$ which will be chosen later. This sum will be small since it contains few summands. 

The second part will consist of all summands where $u>n^{\delta}$. This corresponds to $n_1t_1>\frac{n+n^{\delta}}{m}$ which will force $n_2t_2$ and $n_3t_3$ to be small.

The following Lemma~\ref{lem: Browning Elsholtz f3 bound} is \cite{TheNumber}*{Theorem 2}.
\begin{customlem}B \label{lem: Browning Elsholtz f3 bound}
For any $\epsilon >0$, we have
$$f_3(m,n)\ll_{\epsilon} n^{\epsilon}\left(\frac{n}{m}\right)^{\frac{2}{3}}.$$
\end{customlem}
In the proof of Theorem~\ref{thm: k fractions theorem} below we make use of Lemma~\ref{lem: Browning Elsholtz f3 bound} rather than Theorem~\ref{thm: main theorem}. Furthermore we will use a lifting procedure which was first used by Browning and Elsholtz~\cite{TheNumber} to lift upper bounds of the form
\begin{equation} \label{eq: lifting equation}
f_5(m,n)\ll_{\epsilon} n^{\epsilon}\left(\frac{n^2}{m}\right)^c
\end{equation}
to upper bounds for $f_k(m,n)$ for $k > 5$. For possible future use we write this procedure up in the following lemma and work through the original proof by Browning and Elsholtz with an arbitrary exponent $c>1$ in \eqref{eq: lifting equation}.

\begin{customlem}C \label{lem: Browning Elsholtz lifting lemma}
Suppose that there exists $c>1$ such that
$$f_5(m,n)\ll_{\epsilon} n^{\epsilon}\left(\frac{n^2}{m}\right)^c.$$
Then for any $k \geq 5$ we have
$$f_k(m,n)\ll_{\epsilon} (kn)^{\epsilon}\left(\frac{k^{\nicefrac{4}{3}}n^2}{m}\right)^{c2^{k-5}}.$$
\end{customlem}

\begin{proof}
We will inductively show that for $k \geq 5$ there exists $\Theta_k$ depending on $k$ such that we have 
\begin{equation} \label{eq: k upper bound independent}
f_k(m,n) \ll_{\epsilon} (kn)^{\epsilon}\left(\frac{k^{\Theta_k}n^2}{m}\right)^{c2^{k-5}}
\end{equation}
and we note that this is certainly true for $k=5$ by assumption. The proof works in three steps.

\textbf{1.} Establish an upper bound where the implied constant is allowed to depend on $k$.

For $k \geq 5$ we want to have a bound of the form
\begin{equation} \label{eq: k upper bound dependent}
f_k(m,n) \ll_{k,\epsilon} n^{\epsilon}\left(\frac{n^2}{m}\right)^{c2^{k-5}}
\end{equation}
where the implied constant is allowed to depend on $k$. An upper bound of this type may easily be achieved via \eqref{eq: Browning Elsholtz u equation}. Indeed this bound holds true for $k=5$ by assumption and assuming its existence for $f_k(m,n)$ we find for $f_{k+1}(m,n)$
\begin{align*}
f_{k+1}(m,n) &\ll \sum_{\substack{0 < u \leq kn \\ m|u+n}}f_k\left(u,\frac{n(u+n)}{m}\right)\ll_{k,\epsilon} n^{\epsilon}\left(\frac{n^2}{m}\right)^{c2^{k-4}}\sum_{u =1}^{\infty}\frac{1}{u^{c2^{k-5}}} \\ 
&\ll_{k,\epsilon} n^{\epsilon}\left(\frac{n^2}{m}\right)^{c2^{k-4}},
\end{align*}
where we used that $c>1$.

\textbf{2.} Use inequality \eqref{eq: Browning Elsholtz u equation} and split the sum into two parts.

For the upper bound where the implied constant is independent of $k$ we again suppose it to be true for $f_k(m,n)$ with $k\geq 5$ and inductively prove it to hold for $f_{k+1}(m,n)$. Using inequalities \eqref{eq: Browning Elsholtz u equation} and \eqref{eq: k upper bound independent} we get
\begin{align*}
f_{k+1}(m,n) &\ll \sum_{\substack{0 < u \leq kn \\ m|u+n}}f_k\left(u,\frac{n(u+n)}{m}\right) \\ 
&\ll \sum_{\substack{0 < u \leq (L-1)n \\ m|u+n}}f_k\left(u,\frac{n(u+n)}{m}\right)  + \sum_{\substack{(L-1)n < u \leq kn \\ m|u+n}}f_k\left(u,\frac{n(u+n)}{m}\right) \\
&\ll (kn)^{\epsilon}k^{\Theta_kc2^{k-5}}\left(\frac{n^2}{m}\right)^{c2^{k-4}}\times \\
&\left(\sum_{0 < u \leq (L-1)n}\frac{1}{u^{c2^{k-5}}}L^{c2^{k-4}}+\sum_{(L-1)n<u\leq kn}\frac{1}{u^{c2^{k-5}}}(k+1)^{c2^{k-4}}\right).
\end{align*} 
Since $c2^{k-5}>1$ the infinite sums over $\frac{1}{u^{c2^{k-5}}}$ converge. For the first sum we use that the sum is bounded by a constant for the second sum we use the following more accurate bound
$$\sum_{(L-1)n < u \leq kn}\frac{1}{u^{c2^{k-5}}} \leq \sum_{u=L}^{\infty}\frac{1}{u^{c2^{k-5}}} \ll \int_{L}^{\infty}\frac{1}{u^{c2^{k-5}}}\mathrm{d}u\ll L^{1-c2^{k-5}}.$$
Together with the fact that $(a+b)^{\alpha} \geq a^{\alpha}+b^{\alpha}$ for $a,b > 0$ and $\alpha > 1$ this shows that
\begin{align*}
f_{k+1}&(m,n) \\ 
& \ll_{\epsilon}((k+1)n)^{\epsilon}(k+1)^{\Theta_kc2^{k-5}}\left(\frac{n^2}{m}\right)^{c2^{k-4}}\left(L^{c2^{k-4}}+\left(\frac{k+1}{L^{\nicefrac{1}{2}-(c2^{k-4})^{-1}}}\right)^{c2^{k-4}}\right) \\
&\ll_{\epsilon} ((k+1)n)^{\epsilon}(k+1)^{\Theta_kc2^{k-5}}\left(\frac{n^2}{m}\right)^{c2^{k-4}}\left(L+\frac{k+1}{L^{\nicefrac{1}{2}-(c2^{k-4})^{-1}}}\right)^{c2^{k-4}}.
\end{align*}

\textbf{3.} Optimizing for $L$ and determining an upper bound for $\Theta_k$.

By the bound we derived in step 1 we may suppose that $k \geq \max \{\frac{\log(\frac{2}{3}(c\epsilon)^{-1})}{\log 2}+4,(\frac{1+\sqrt{5}}{2})^{\nicefrac{1}{\epsilon}}-1\}$. With $L=(k+1)^{\nicefrac{2}{3}}$ we get
\begin{align*}
&f_{k+1}(m,n) \\
&\ll_{\epsilon} ((k+1)n)^{\epsilon}(k+1)^{\Theta_kc2^{k-5}}\left(\frac{n^2}{m}\right)^{c2^{k-4}}(k+1)^{\nicefrac{2}{3}\cdot c2^{k-4}}\left(1+L^{(c2^{k-4})^{-1}}\right)^{c2^{k-4}} \\
&\ll_{\epsilon} (k+1)^{\epsilon(1+c2^{k-3})}n^{\epsilon}(k+1)^{c2^{k-4}(\nicefrac{\Theta_k}{2}+\nicefrac{2}{3})}\left(\frac{n^2}{m}\right)^{c2^{k-4}}.
\end{align*}
With $\Theta_{k+1}=\frac{\Theta_k}{2}+\frac{2}{3}$ and an appropriate choice of $\epsilon$ this implies
$$f_{k+1} \ll_{\epsilon} ((k+1)n)^{\epsilon}\left(\frac{(k+1)^{\Theta_{k+1}}n^2}{m}\right)^{c2^{(k+1)-5}}$$
Since for $\Theta_5 \leq \frac{4}{3}$ the sequence recursively defined by $\Theta_{k+1}=\frac{\Theta_k}{2}+\frac{2}{3}$ monotonically increases towards its limit $\frac{4}{3}$ we eventually get for any $k \geq 5$:
$$f_k(m,n)\ll_{\epsilon} (kn)^{\epsilon}\left(\frac{k^{\nicefrac{4}{3}}n^2}{m}\right)^{c2^{k-5}}.$$
\end{proof}

\begin{proof}[Proof of Theorem~\ref{thm: k fractions theorem}.]
In the following $\delta<1$ is a fixed constant to be chosen at the end of the proof. We start with proving bounds on $f_4(m,n)$ and we write $f_4(m,n)=f_4^{(1)}(m,n)+f_4^{(2)}(m,n)$. Here $f_4^{(1)}(m,n)$ counts those solutions of equation \eqref{eq: k fractions pattern equation} with $n_1t_1 \leq \frac{n+n^{\delta}}{m}$ and $f_4^{(2)}(m,n)$ those with $n_1t_1 > \frac{n+n^{\delta}}{m}$. From \eqref{eq: Browning Elsholtz u equation} we have that
\begin{align*}
f_4(m,n) &= f_4^{(1)}(m,n)+f_4^{(2)}(m,n) \leq \sum_{\substack{0 < u \leq n^{\delta} \\ m|u+n}}f_{3}\left(u,\frac{n(u+n)}{m}\right)+f_4^{(2)}(m,n) \\
&=S_1+f_4^{(2)}(m,n).
\end{align*}
We use the following estimate (uniform in $a \in \mathbb{Z}$)
\begin{equation} \label{eq: Browning Elsholtz geometric sum bound}
\sum_{\substack{n \leq x \\ n \equiv a \bmod q}} n^{- \Theta}=\frac{x^{1-\Theta}}{(1+\Theta) q} + \mathcal{O}_{\Theta}(1).
\end{equation} 
To bound the sum $S_1$ we use \eqref{eq: Browning Elsholtz geometric sum bound} and Lemma~\ref{lem: Browning Elsholtz f3 bound} to get
\begin{equation}\label{eq: geometric sum bound consequence}
S_1 \ll_{\epsilon} n^{\epsilon}\left(\frac{n^2}{m}\right)^{\frac{2}{3}}\sum_{\substack{0 < u \leq n^{\delta} \\ m|u+n}}\frac{1}{u^{\frac{2}{3}}} \ll_{\epsilon} n^{\epsilon}\left(\frac{n^2}{m}\right)^{\frac{2}{3}}\left(\frac{n^{\frac{\delta}{3}}}{m}+1\right). 
\end{equation}
Next we prove that
$$f_4^{(2)}(m,n) \ll_{\epsilon} n^{\epsilon}\frac{n^{\nicefrac{(12-4\delta)}{5}}}{m^{\nicefrac{8}{5}}}.$$
Since there are at most $\mathcal{O}_{\epsilon}(n^{\epsilon})$ distinct patterns $(n_1,n_2,n_3,n_4)$ it suffices to prove this bound for all solutions counted by $f_4^{(2)}(m,n)$ corresponding to a fixed pattern. To get an upper bound for the contribution of $f_4^{(2)}(m,n)$ we thus suppose that $(n_1,n_2,n_3,n_4)$ is fixed and note that the fact that $\frac{4n}{m} \geq n_1t_1 > \frac{n+n^{\delta}}{m}$ implies the following upper bound for $n_2t_2$:
$$\frac{3}{n_2t_2} \geq \frac{mn_1t_1-n}{nn_1t_1} \geq \frac{mn^{\delta}}{4n^2}.$$
Therefore we have
\begin{equation} \label{eq: improved t2 bound}
n_2t_2 \ll \frac{n^{2-\delta}}{m}.
\end{equation}
We use again relative greatest common divisors and write a representation of $\frac{m}{n}$ as a sum of four unit fractions as
\begin{align*}
\frac{m}{n}&=\frac{1}{n_1x_1x_{12}x_{13}x_{14}x_{123}x_{124}x_{134}x_{1234}}+\frac{1}{n_2x_2x_{12}x_{23}x_{24}x_{123}x_{124}x_{234}x_{1234}}\\
&+\frac{1}{n_3x_3x_{13}x_{23}x_{34}x_{123}x_{134}x_{234}x_{1234}}+\frac{1}{n_4x_4x_{14}x_{24}x_{34}x_{124}x_{134}x_{234}x_{1234}}.
\end{align*}
It is again easy to see that $x_1=x_2=x_3=x_4=1$ and multiplying out the last equation yields
\begin{equation} \label{eq: multiplied out 4 terms equation}
\begin{split} 
m&x_{12}x_{13}x_{14}x_{23}x_{24}x_{34}x_{123}x_{124}x_{134}x_{234}x_{1234} \\
&=\frac{n}{n_1}x_{23}x_{24}x_{34}x_{234}+\frac{n}{n_2}x_{13}x_{14}x_{34}x_{134}+\frac{n}{n_3}x_{12}x_{14}x_{24}x_{124}+\frac{n}{n_4}x_{12}x_{13}x_{23}x_{123}.
\end{split}
\end{equation}
From equation \eqref{eq: multiplied out 4 terms equation} we see that the quantity
$$z_{34}=\frac{\frac{n}{n_3}x_{12}x_{14}x_{24}x_{124}+\frac{n}{n_4}x_{12}x_{13}x_{23}x_{123}}{x_{34}}$$
is an integer and we use
\begin{equation} \label{eq: z34 equation}
z_{34}x_{34}=\frac{n}{n_3}x_{12}x_{14}x_{24}x_{124}+\frac{n}{n_4}x_{12}x_{13}x_{23}x_{123}.
\end{equation}
By \eqref{eq: improved t2 bound} and $\frac{4n}{m} \geq n_1t_1 > \frac{n+n^{\delta}}{m}$ we have
\begin{equation} \label{eq: upper bound all 4 factors}
(t_1t_2)^4=(x_{12}x_{13}x_{14}x_{123}x_{124}x_{134}x_{1234})^4(x_{12}x_{23}x_{24}x_{123}x_{124}x_{234}x_{1234})^4 \ll \frac{n^{12-4\delta}}{m^8},
\end{equation}
and we write
\begin{equation} \label{eq: 4 terms factors to consider}
\begin{split}
&(x_{12}x_{13}x_{14}x_{123}x_{124}x_{134}x_{1234})^4(x_{12}x_{23}x_{24}x_{123}x_{124}x_{234}x_{1234})^4 = \\ 
&(x_{12}x_{13}x_{14}x_{23}x_{24}x_{123}x_{124}x_{1234})(x_{12}x_{13}x_{23}x_{24}x_{123}x_{124}x_{134}x_{234}x_{1234})\times \\
&(x_{12}x_{14}x_{23}x_{24}x_{123}x_{124}x_{134}x_{234}x_{1234})(x_{12}x_{13}x_{14}x_{24}x_{123}x_{124}x_{134}x_{234}x_{1234})\times \\
&(x_{12}^4x_{13}x_{14}x_{23}x_{123}^4x_{124}^4x_{134}x_{234}x_{1234}^4).
\end{split}
\end{equation}
We show that each of the five factors in brackets on the right hand side of the last equation corresponds to at most $\mathcal{O}_{\epsilon}(n^{\epsilon})$ solutions of \eqref{eq: multiplied out 4 terms equation}, where $\epsilon$ is an arbitrarily small positive number. First we note that all factors are of polynomial size in $n$ and by Lemma~\ref{lem: divisor bound}, given one of these factors, we have $\mathcal{O}_{\epsilon}(n^{\epsilon})$ choices for all the $x_{ij}$, $x_{ijk}$ and $x_{1234}$ appearing as sub-factors.

Given positive integer constants $C_0,C_1,C_2$ and $C_3$ of size polynomial in $n$, we count the number of integer solutions $(A,B)$ of the equation
\begin{equation} \label{eq: Elsholtz Browning factorization trick}
C_0AB=C_1A+C_2B+C_3.
\end{equation}
Rewriting this equation in the form
$$(C_0A-C_2)(C_0B-C_1)=C_0C_3+C_1C_2$$
we see that the number of solutions $(A,B)$ is bounded by $\mathcal{O}_{\epsilon}(n^{\epsilon})$. For the second to the fifth factor on the right hand side of \eqref{eq: 4 terms factors to consider} exactly two parameters are missing to uniquely determine a solution of \eqref{eq: multiplied out 4 terms equation}. All of these factors miss the parameter $x_{34}$. The second one additionally misses $x_{14}$, the third one $x_{13}$, the fourth one $x_{23}$ and the last one $x_{24}$. In all of these cases equation \eqref{eq: multiplied out 4 terms equation} provides an instance of \eqref{eq: Elsholtz Browning factorization trick} where the variables $A$ and $B$ correspond to the two missing parameters (the term containing both missing parameters on the right hand side of \eqref{eq: multiplied out 4 terms equation} may be shifted to the left hand side). 

In the first factor on the right hand side of \eqref{eq: 4 terms factors to consider} three parameters are missing. From equation \eqref{eq: z34 equation} we see that we have at most $\mathcal{O}_{\epsilon}(n^{\epsilon})$ choices for the parameter $x_{34}$. To see the same bound for the parameters $x_{134}$ and $x_{234}$ we use again that equations of type \eqref{eq: Elsholtz Browning factorization trick} can be factorized.

Since by \eqref{eq: upper bound all 4 factors} at least one of the factors on the right hand side of \eqref{eq: 4 terms factors to consider} is $\mathcal{O}\left(\frac{n^{\nicefrac{(12-4\delta)}{5}}}{m^{\nicefrac{8}{5}}}\right)$ we have that
\begin{equation} \label{eq: f4 delta bound}
f_4^{(2)}(m,n) \ll_{\epsilon} n^{\epsilon}\frac{n^{\nicefrac{(12-4\delta)}{5}}}{m^{\nicefrac{8}{5}}}.
\end{equation}
Again we note that in the considerations above the divisor bound from Lemma~\ref{lem: divisor bound} was applied a bounded number of times and the bound in \eqref{eq: f4 delta bound} follows upon redefining the choice of $\epsilon$. Choosing $\delta=\frac{16}{17}$ in \eqref{eq: geometric sum bound consequence} and \eqref{eq: f4 delta bound} we get
\begin{equation} \label{eq: f4 final bound}
f_4(m,n) \ll n^{\epsilon}\left(\frac{n^{\nicefrac{4}{3}}}{m^{\nicefrac{2}{3}}}+\frac{n^{\nicefrac{28}{17}}}{m^{\nicefrac{8}{5}}}\right).
\end{equation}
To bound $f_5(m,n)$ we again use \eqref{eq: Browning Elsholtz u equation} and \eqref{eq: Browning Elsholtz geometric sum bound} and get
\begin{equation} \label{eq: f5 final bound}
f_5(m,n)\ll n^{\epsilon}\sum_{\substack{0 < u \leq 4 n \\ m|u+n}}\left(\left(\frac{n^2}{m}\right)^{\nicefrac{4}{3}}\frac{1}{u^{\nicefrac{2}{3}}}+\left(\frac{n^2}{m}\right)^{\nicefrac{28}{17}}\frac{1}{u^{\nicefrac{8}{5}}}\right) \ll n^{\epsilon}\left(\frac{n^2}{m}\right)^{\nicefrac{28}{17}}.
\end{equation}

Setting $c=\frac{28}{17}$ in Lemma~\ref{lem: Browning Elsholtz lifting lemma} yields the bound in Theorem~\ref{thm: k fractions theorem}.
\end{proof}

\section{Lower bounds}

\begin{proof}[Proof of Theorem~\ref{thm: lower bounds theorem}]
To prove the first bound we are going to extend an idea used in the proof of~\cite{TheNumber}*{Theorem 1}. As before we use highly composite denominators $n\in \mathbb{N}$, but here we show that there are many values $a_1$ with many corresponding pairs $(a_2,a_3)$ giving a solution of 
$$\frac{m}{n}=\frac{1}{a_1}+\frac{1}{a_2}+\frac{1}{a_3}.$$
To prove our lower bound for $f_3(m,n)$ we consider the set 
$$\mathcal{N}=\left\{mn^{\prime}:n^{\prime}=\prod_{i=1}^rp_i\right\},$$
where $p_i$ is the $i$-th prime. In choosing the denominators $n \in \mathcal{N}$ we reduce the problem to finding many solutions of the equation
$$\frac{1}{n^{\prime}}=\frac{1}{a_1}+\frac{1}{a_2}+\frac{1}{a_3}.$$
We set $a_1=n^{\prime}+d$, where $d$ is any divisor of $n^{\prime}$, and are left with
$$\frac{1}{n^{\prime}}-\frac{1}{n^{\prime}+d}=\frac{1}{n^{\prime}\left(\frac{n^{\prime}}{d}+1\right)}=\frac{1}{a_2}+\frac{1}{a_3}.$$
For two divisors $d_1$ and $d_2$ of $n^{\prime}$ with $(d_1,d_2)=1$ we have
\begin{equation} \label{eq: 3 lower bound equation 2}
\frac{1}{n^{\prime}\left(\frac{n^{\prime}}{d}+1\right)}=\frac{1}{\frac{n^{\prime}\left(\frac{n^{\prime}}{d}+1\right)}{d_1}(d_1+d_2)}+\frac{1}{\frac{n^{\prime}\left(\frac{n^{\prime}}{d}+1\right)}{d_2}(d_1+d_2)}.
\end{equation}
We note that for two pairs of divisors $d_1,d_2$ and $d_1^{\prime},d_2^{\prime}$ with $(d_1,d_2)=1$ and $(d_1^{\prime},d_2^{\prime})=1$ it follows that
$$\frac{n^{\prime}\left(\frac{n^{\prime}}{d}+1\right)}{d_1}(d_1+d_2)=\frac{n^{\prime}\left(\frac{n^{\prime}}{d}+1\right)}{d_1^{\prime}}(d_1^{\prime}+d_2^{\prime})\Leftrightarrow \frac{d_1}{d_2}=\frac{d_1^{\prime}}{d_2^{\prime}}.$$
Since $d_1$ and $d_2$ as well as $d_1^{\prime}$ and $d_2^{\prime}$ are coprime we get $d_1=d_1^{\prime}$ and $d_2=d_2^{\prime}$. This implies that each pair $(d_1,d_2)$ with $d_1<d_2$ gives a unique solution of equation \eqref{eq: 3 lower bound equation 2}. Furthermore for any choice of $d,d_1,d_2$ it follows that 
$$n^{\prime}+d<\frac{n^{\prime}\left(\frac{n^{\prime}}{d}+1\right)}{d_1}(d_1+d_2),$$
which altogether implies that by counting all possible choices for $d,d_1,d_2$
we get a lower bound for twice the value of $f_3(1,n^{\prime})$.

Choosing $n^{\prime}$ as in the construction of the set $\mathcal{N}$, we have $2^{\omega(n^{\prime})}$ choices for the divisor $d$ and using the binomial theorem there are
$$\sum_{i=0}^{\omega(n^{\prime})}\binom{\omega(n^{\prime})}{i}\sum_{j=0}^{\omega(n^{\prime})-i}\binom{\omega(n^{\prime})-i}{j}=\sum_{i=0}^{\omega(n^{\prime})}\binom{\omega(n^{\prime})}{i}2^{\omega(n^{\prime})-i}=3^{\omega(n^{\prime})}$$
choices for the divisors $d_1$ and $d_2$. As a consequence of the prime number theorem it is known that $\omega(n^{\prime}) \sim \frac{\log n^{\prime}}{\log \log n^{\prime}}$ and hence, for $n\in \mathcal{N}$
\begin{align*}
f_3(m,n)=f_3(1,n^{\prime}) \geq \frac{1}{2}2^{\omega(n^{\prime})}3^{\omega(n^{\prime})} &\geq \exp\left((\log 6+o(1))\frac{\log n^{\prime}}{\log \log n^{\prime}}\right)\\
&\geq \exp\left((\log 6+o_m(1))\frac{\log n}{\log \log n}\right).
\end{align*}

For the second bound we modify the idea used in the proof of~\cite{CountingThe}*{Theorem 1.8}. For fixed $m \in \mathbb{N}$, as a consequence of the Tur{\'a}n-Kubilius inequality (see e.g. \cite{IntroductionA}*{p. 434}) we get that the set
$$\mathcal{M}_1=\bigcap_{\substack{k \leq m \\ (k,m)=1}}\left\{n \in \mathbb{N}:  \omega(n,k,m)=\left(\frac{1}{\varphi(m)}+o(1)\right)\log\log n \right\}$$
is a set with density one, i.e. $\lim_{x \rightarrow \infty}\frac{\{n \in \mathcal{M}_1:n\leq x\}}{x}=1$.

For any $n \in \mathcal{M}_1$ we write $\frac{m}{n}=\frac{m^{\prime}}{n^{\prime}}$ with $(m^{\prime},n^{\prime})=1$ and note that $\omega(n,k,m)=\omega(n^{\prime},k,m)$ for all $k$ with $(k,m)=1$. By construction of the set $\mathcal{M}_1$ and since $n^{\prime}$ is coprime to $m^{\prime}$, we find $\left(\frac{1}{\varphi(m)}+o(1)\right)\log \log n$ prime divisors $p$ of $n^{\prime}$ in the residue class $-n^{\prime}\bmod m^{\prime}$. For any of these prime divisors we have
$$\frac{m^{\prime}}{n^{\prime}}-\frac{1}{\frac{n^{\prime}+p}{m^{\prime}}}=\frac{p}{n^{\prime}\frac{n^{\prime}+p}{m^{\prime}}}=\frac{1}{n^{\prime}\frac{\nicefrac{n^{\prime}}{p}+1}{m^{\prime}}}$$
where $\frac{\nicefrac{n^{\prime}}{p}+1}{m^{\prime}}$ is an integer. Again, by construction of the set $\mathcal{M}_1$, for the number of prime factors of $n^{\prime}$ we have
$$\omega(n^{\prime}) \geq \omega(n)-\omega(m) = (1+o_m(1))\log \log n.$$

For two coprime divisors $d_1$ and $d_2$ of $n^{\prime}$ we construct decompositions of $\frac{1}{n^{\prime}\frac{\nicefrac{n^{\prime}}{p}+1}{m^{\prime}}}$ as a sum of two unit fractions as in \eqref{eq: 3 lower bound equation 2}. As above we see that for any prime divisor $p$ of $n^{\prime}$ in the residue class $-n^{\prime} \bmod m^{\prime}$ there are at least $3^{\omega(n^{\prime})}$ such decompositions and all of them are distinct.

Altogether this implies that for any $n \in \mathcal{M}_1$
\begin{align*}
f(m,n) &\geq \left(\frac{1}{\varphi(m)}+o(1)\right)3^{\omega(n^{\prime})} \cdot \log \log n \geq \left(\frac{1}{\varphi(m)}+o(1)\right)3^{\omega(\nicefrac{n}{m})} \cdot \log \log n \\
&\geq \exp((\log 3 + o_m(1))\log \log n)\cdot \log \log n.
\end{align*}

Finally, we prove the improved lower bound on $f_3(4,n)$. To do so, we set
\begin{align*}
\mathcal{M}_2 &=\left(\bigcap_{i \in \{1,3\}}\{n \in \mathbb{N}:\frac{\tau(n,4)}{4}\leq \tau(n,i,4)\}\right) \cap \\ 
&\cap \{n \in \mathbb{N}:\omega(n)=(1+o(1))\log \log n\} \cap \{n \in \mathbb{N}: \tau(n) \geq (\log n)^{\log 2+o(1)}\}.
\end{align*}
The first two sets with $i=1$ and $i=3$ in the intersection in the definition of $\mathcal{M}_2$ have density $1$ by~\cite{OnTheDistribution}*{Theorem 5}. For the third and the fourth set this is true by the Tur{\'a}n-Kubilius inequality (again see e.g. \cite{IntroductionA}*{p. 434}). Hence the set $\mathcal{M}_2$ has density $1$ and we investigate what happens for $n$ in a certain residue class modulo $4$.

If $n \equiv 0 \bmod 4$, then $\frac{4}{n}=\frac{1}{\nicefrac{n}{4}}$ and for any divisor $d$ of $\frac{n}{4}$ we have
$$\frac{1}{\frac{n}{4}}-\frac{1}{\frac{n}{4}+d}=\frac{1}{\frac{n}{4}\left(\frac{n}{4d}+1\right)}.$$ 
Since $\omega\left(\frac{n}{4}\right)\geq \omega(n)-1$, with the same arguments as above, we conclude that the number of representations of $\frac{1}{\nicefrac{n}{4}\left(\nicefrac{n}{4d}+1\right)}$ as a sum of two unit fractions is at least of order $3^{\omega(\nicefrac{n}{4})}=3^{(1+o(1))\log \log n}$. From $\tau(n)=\prod_{p|n}(\nu_p(n)+1)$ we easily deduce that $\tau\left(\frac{n}{4}\right)\geq \frac{1}{3}\tau(n)$. Altogether we thus get
$$f_3(4,n)\geq \frac{1}{3} \tau\left(\frac{n}{4}\right)3^{\omega(\nicefrac{n}{4})} \geq \exp((\log 6 +o(1))\log \log n).$$

If $n \equiv 2 \bmod 4$, then $\frac{n}{2}$ is odd and the same is true for all $\tau\left(\frac{n}{2}\right)=\frac{1}{2}\tau(n)$ divisors of $\frac{n}{2}$. We have $\frac{4}{n}=\frac{2}{\nicefrac{n}{2}}$ and for any divisor $d$ of $\frac{n}{2}$
$$\frac{2}{\frac{n}{2}}-\frac{1}{\frac{\nicefrac{n}{2}+d}{2}}=\frac{1}{\frac{n}{2}\left(\frac{\nicefrac{n}{2d}+1}{2}\right)}.$$
As above we get
$$f_3(4,n)\geq \tau\left(\frac{n}{2}\right)3^{\omega(n)-1} \geq \exp((\log 6 + o(1))\log \log n).$$

Finally, if $n\equiv r \bmod 4$ for $r \in \{1,3\}$, we have $\tau(n,4)=\tau(n)$ and by construction of the set $\mathcal{M}_2$, we have more than $\frac{\tau(n)}{4}$ divisors $d$ of $n$ in the residue class $-r \bmod 4$. Again, for any of these divisors we have
$$\frac{4}{n}-\frac{1}{\frac{n+d}{4}}=\frac{1}{n\left(\frac{\nicefrac{n}{d}+1}{4}\right)}.$$
Applying the arguments used previously one more time, we find
$$f_3(4,n)\geq \frac{\tau(n)}{4}3^{\omega(n)} \geq \exp((\log 6 +o(1))\log \log n)$$
also in this case.

\end{proof}

\begin{remark}
The difference in the constants in the exponential functions of the lower bounds on $f(m,n)$ and $f(4,n)$ for sets of integers with density one in Theorem~\ref{thm: lower bounds theorem} is basically due to cancellation effects when dealing with general $m$. In particular we deal with $\frac{m}{n}=\frac{m^{\prime}}{n^{\prime}}$, where $(m^{\prime},n^{\prime})=1$, and we would need to have good control of the number of divisors of $n^{\prime}$ in the residue class $-n^{\prime} \bmod m^{\prime}$ to get the $\log 6$ exponent also in the general case. However, if we do not ask about a lower bound holding for a set of density one within the positive integers, but for a set of integers of density one within the set $\mathcal{S}$ of positive integers coprime to a given $m \in \mathbb{N}$, we may achieve the $\log 6$ exponent. To do so we replace the set $\mathcal{M}_1$ with
\begin{align*}
\mathcal{M}_1^{\prime} &=\left(\bigcap_{\substack{1\leq i \leq m \\ (i,m)=1}}\{n \in \mathbb{N}:\tau(n,i,m)=\frac{\tau(n)}{\varphi(m)}(1+o_m(1))\}\right) \cap \\ 
&\cap \{n \in \mathbb{N}:\omega(n)=(1+o(1))\log \log n\} \cap \{n \in \mathbb{N}: \tau(n) \geq (\log n)^{\log 2+o(1)}\} \cap \mathcal{S}.
\end{align*}
Now we may use results from~\cite{OnTheDistribution}*{Theorem 5} as well as Tur{\'a}n-Kubilius like previously and get that $\mathcal{M}_1^{\prime}$ has density one in $\mathcal{S}$. Instead of constructing the first denominator via shifts in prime factors of $n$ we may use arbitrary divisors of $n$ in this case, which leads to the improvement mentioned above.
\end{remark}

\begin{proof}[Proof of Theorem~\ref{thm: lower bound for primes theorem}]
We consider solutions corresponding to the pattern $(1,p,p)$. In equation~\eqref{eq: 3 fractions main equation} we suppose that $a_1$ is the denominator with $(a_1,p)=1$ and we write $a_1=t_1$, $a_2=pt_2$ and $a_3=pt_3$. We use the parametrization via relative greatest common divisors of the $t_i$ and applying Lemma~\ref{lem: rgcd lemma} it is easy to see, that $x_1=x_2=x_3=1$ in this case. Hence we are looking for infinitely many primes $p \equiv e \bmod f$ such that for given $m \in \mathbb{N}$ the equation
\begin{equation} \label{eq: lower bound for primes equation 1}
\frac{m}{p}=\frac{1}{x_{12}x_{13}x_{123}}+\frac{1}{px_{12}x_{23}x_{123}}+\frac{1}{px_{13}x_{23}x_{123}}
\end{equation}
has many solutions. Multiplying equation~\eqref{eq: lower bound for primes equation 1} by the common denominator we get
$$mx_{12}x_{13}x_{23}x_{123}=px_{23}+x_{13}+x_{12}.$$ 
Setting $x_{12}+x_{13}=kx_{23}$, $M=\lcm(m,f)$ and $x_{12}=\frac{M}{m}$ we deduce that
$$M\left(kx_{23}-\frac{M}{m}\right)x_{123}=p+k.$$
The residue class $(f-e)\equiv-e \bmod f$ splits into the residue classes $(f-e)+if \bmod M$, for $0 \leq i \leq \frac{m}{(m,f)}-1$. Note, that $\gcd\left(f,\frac{m}{(m,f)}\right)=1$ hence the integers $i\cdot f$ for $0\leq i \leq \frac{m}{(m,f)}-1$ are a full system of residues modulo $\frac{m}{(m,f)}$. In particular there exists a $0\leq j \leq \frac{m}{(m,f)}-1$ such that $(f-e)+jf \equiv 1 \bmod \frac{m}{(m,f)}$. We set $k=(f-e)+jf$ and with $(e,f)=1$ we altogether see that $(M,k)=1$.

Now let $Q=\prod_{i=1}^rq_i$ where $q_i$ is the $i$-th prime with $q_i \equiv -\frac{M}{m}\bmod k$ and $q_i>M$. Note that $\gcd(M,Q)=1$.

With $r=\left\lfloor \frac{\log t}{\varphi(k) C\log \log t} \right\rfloor$ we find that $Q$ is of order $t^{\nicefrac{1}{C}+o_{f,m}(1)}$. We now use Linnik's theorem on primes in arithmetic progressions. As the modulus is very smooth we can use an exponent of $C=\frac{12}{5}+o(1)$, due to Chang~\cite{ShortCharacter}*{Corollary 11}. Hence we may find a prime $p$ of order $M^Ct^{1+o_{f,m}(1)}$ with
$$p \equiv -k \bmod QM.$$
This congruence implies that $p+k$ is divisible by the primes $q_1, \ldots, q_r$ and together with $k=(f-e)+jf$, we deduce that $p \equiv e \bmod f$ and $p+k \equiv 0 \bmod M$.

Let $l \in \mathbb{N}_0$ and $S$ be a subset of size $l\ord_k\left(-\frac{M}{m}\right)+1$ of the prime factors of $Q$. Hence $x_{23}=\frac{\prod_{q\in S}q+\frac{M}{m}}{k}$ is an integer and we set $x_{123}=\frac{p+k}{M\prod_{q\in S}q}$. We observe that any of these choices leads to a different solution of \eqref{eq: lower bound for primes equation 1}. To see this we look at the denominator $a_2=px_{12}x_{23}x_{123}$ of the second fraction on the right hand side of this equation.
Suppose that two sets $S$ and $S^{\prime}$ would lead to the same denominator $a_2$. With $x_{12}=\frac{M}{m}$ this would imply the existence of $x_{23} \neq x_{23}^{\prime}$ such that
$$p\frac{M}{m}x_{23}\frac{p+k}{M(kx_{23}-\frac{M}{m})}=p\frac{M}{m}x_{23}^{\prime}\frac{p+k}{M(kx_{23}^{\prime}-\frac{M}{m})}$$
from which we derive that
$$\frac{x_{23}}{x_{23}^{\prime}}=\frac{kx_{23}-\frac{M}{m}}{kx_{23}^{\prime}-\frac{M}{m}}=\frac{\prod_{q\in S}q}{\prod_{q^{\prime}\in S^{\prime}}q^{\prime}}.$$
If $q\in S$ would divide $x_{23}$ then $q$ would also divide $\frac{M}{m}$, which is impossible by construction of $Q$. We hence have that $\frac{\prod_{q\in S}q}{\prod_{q^{\prime}\in S^{\prime}}q^{\prime}}=1$ and thus $S=S^{\prime}$. 

To count the number of solutions we get with the above construction, we make use of a formula which can be found in~\cite{SumsOfEvenly}*{Theorem 1}, for example, and which states
\begin{equation} \label{eq: binomial sum formula}
\sum_{i\geq 0}\binom{n}{iu} = \frac{1}{u}\sum_{j=0}^{u-1}(1+\xi_u^j)^n,
\end{equation}
where $\xi_u=\exp\left(\frac{2\pi i}{u}\right)$. Note that for the term corresponding to $j=0$ in the sum on the right hand side of ~\eqref{eq: binomial sum formula} we get $2^n$ while for all other $j$ we have $|1+\xi_u^j|<2$. Hence we deduce
$$\sum_{i\geq 0}\binom{n}{iu}=\frac{2^n}{u}(1+o_u(1)).$$
The number of choices of the parameter $x_{23}$ is
\begin{align*}
\sum_{i\geq 0}&\binom{r}{i\ord_k\left(-\frac{M}{m}\right)+1}=\sum_{i\geq 0}\binom{r+1}{i\ord_k\left(-\frac{M}{m}\right)}-\sum_{i\geq 0}\binom{r}{i\ord_k\left(-\frac{M}{m}\right)}\\
&=\frac{2^{r+1}}{\ord_k\left(-\frac{M}{m}\right)}(1+o_{f,m}(1))-\frac{2^r}{\ord_k\left(-\frac{M}{m}\right)}(1+o_{f,m}(1)) \\
&=\frac{2^r}{\ord_k\left(-\frac{M}{m}\right)}(1+o_{f,m}(1)).
\end{align*}
Plugging in $r=\left\lfloor \frac{\log t}{\varphi(k) C\log \log t} \right\rfloor$ and using that $p \leq M^Ct^{1+o_{f,m}(1)}$ we get a lower bound of 
\begin{equation} \label{eq: lower bound primes last equation}
\begin{aligned}
f_3(m,p) &\gg_{f,m} \exp\left(\left(\frac{\log 2}{C \varphi(k)}+o_{f,m}(1)\right)\frac{\log t}{\log \log t}\right) \\
&\gg_{f,m} \exp\left(\left(\frac{5\log 2}{12 \lcm(m,f)}+o_{f,m}(1)\right)\frac{\log p}{\log \log p}\right).
\end{aligned}
\end{equation}
\end{proof}

\begin{remark}
The best known exponent for Linnik's Theorem takes care of the worst case modulus and is $5$ by work of Xylouris~\cite{UeberDieNullstellen}. Chang's result~\cite{ShortCharacter}*{Corollary 11} considers smooth moduli (as in our situation) and allows for the better exponent $\frac{12}{5}$. Harman investigated, in connection with constructing Carmichael numbers, what happens if one is allowed to avoid a small set of exceptional moduli. In this situation he improved the exponent to $\frac{1}{0.4736}$ (see~\cite{WattsMean}*{Theorem 1.2} and~\cite{OnTheNumberOfCarmichael} for some more explanation). As in our situation we choose the modulus $M$, and hence can avoid "bad"   factors, it seems possible that Theorem~\ref{thm: lower bound for primes theorem} can also be proved with a factor of $0.4736$ instead of $\frac{5}{12}\approx 0.4167$ in the exponent of the lower bound on $f_3(m,p)$.
\end{remark}

\begin{remark}
If we consider the case $m=4$, $f=4$ and $e\in \{1,3\}$ in Theorem~\ref{thm: lower bound for primes theorem}, we can explicitly compute $k$ in the first line of \eqref{eq: lower bound primes last equation}. We simply have $k=3$ if $e=1$ and $k=1$ if $e=3$ hence we arrive at the lower bounds
$$f_3(4,p)\gg \exp\left((0.1444+o(1))\frac{\log p}{\log \log p}\right)$$
if $e=1$ and
$$f_3(4,p)\gg \exp\left((0.2888+o(1))\frac{\log p}{\log \log p}\right)$$
if $e=3$.
\end{remark}

\section*{Acknowledgement}

The authors acknowledge the support of the Austrian Science Fund (FWF): W1230. Furthermore we would like to thank Igor Shparlinski and Glyn Harman for drawing our attention to the papers~\cite{ShortCharacter} and~\cites{OnTheNumberOfCarmichael,WattsMean}.

\end{document}